\documentclass{article}
\usepackage{amsmath,amssymb,amsthm}
\usepackage[utf8]{inputenc}
\usepackage[dvipsnames]{xcolor}
\usepackage{pgf,tikz,color}
\usepackage{multirow}
\usepackage[framemethod=TikZ]{mdframed}
\usepackage[normalem]{ulem}

\usepackage[margin=25mm]{geometry}

\usepackage{hyperref}
\hypersetup{colorlinks = true, citecolor = NavyBlue, linkcolor = NavyBlue, urlcolor = NavyBlue}

\usepackage{cleveref}

\newtheorem{theorem}{Theorem}[section]
\newtheorem{prop}[theorem]{Proposition}
\newtheorem{cor}[theorem]{Corollary}
\Crefname{cor}{Corollary}{Corollaries}
\newtheorem{lemma}[theorem]{Lemma}

\theoremstyle{definition}
\newtheorem{defin}[theorem]{Definition}
\newtheorem{remark}[theorem]{Remark}

\Crefname{algo}{Algorithm}{Algorithms}

\def\cC{\mathcal C}

\def\cH{\mathcal H}

\def\cT{\mathcal T}
\def\cP{\mathcal P}

\def\Z{\mathbb Z}

\def\PG{\mathrm{PG}}

\def\Fq{\mathbb{F}_q}

\DeclareRobustCommand{\ZpZ}[1]{\ifthenelse{\equal{#1}{1}}{\Z/p\Z}{{(\Z/p\Z)^{#1}}}}

\newcommand{\qbinom}[2]{\left[{{#1}\atop#2}\right]_q}

\def\PGO+{\mathrm{PGO}^+}
\def\PGO-{\mathrm{PGO}^-}
\def\PGO{\mathrm{PGO}}

\title{Counting substructures in finite geometries}
\author{Sam Mattheus, Geertrui Van de Voorde}
\date{}

\begin{document}

\maketitle

\begin{abstract}We use techniques from algebraic and extremal combinatorics to derive upper bounds on the number of independent sets in several (hyper)graphs arising from finite geometry. In this way, we obtain asymptotically sharp upper bounds for partial ovoids and EKR-sets of flags in polar spaces, line spreads in $\PG(2r-1,q)$ and plane spreads in $\PG(5,q)$, and caps in $\PG(3,q)$. The latter result extends work due to Roche-Newton and Warren \cite{Roche-Newton/Warren} and Bhowmick and Roche-Newton \cite{Bhowmick/Roche-Newton}. \\Finally, we investigate caps in $p$-random subsets of $\PG(r,q)$, which parallels recent work for arcs in projective planes by Bhowmick and Roche-Newton, and Roche-Newton and Warren \cite{Bhowmick/Roche-Newton,Roche-Newton/Warren}, and arcs in projective spaces by Chen, Liu, Nie and Zeng \cite{CLNZ}. \end{abstract}

\section{Introduction}
A classical problem in finite geometry is the study of interesting substructures such as {\em arcs, caps, spreads,} etc.\ Traditionally, the focus of this research is on constructing, characterising and classifying those objects. Less attention is paid to {\em counting} the number of such substructures. There are a few such results known though: an example of research in this direction is Kantor's count of commutative semifield planes of even order. It should be noted that in this result, the number of such spreads are counted {\em up to isomorphism}. The main result shows that (under some mild hypotheses on $q$ and $m$) the number of such planes of even order $q^m$ is at least $(q^{m(\phi(m)-1)})/(m\log(q))^2)$, which indicates that their classification is a hopeless task.

On the other hand, finding upper bounds on the number of certain substructures of (hyper)graphs is a very common theme in the area of extremal combinatorics. Recently, techniques from extremal combinatorics have been employed to count {\em arcs} in projective planes and higher-dimensional spaces. This problem was addressed in \cite{Bhowmick/Roche-Newton,CLNZ,MubayiVerstraete,Roche-Newton/Warren}. Note that in this case, arcs are not counted up to isomorphism, but are considered different if their point sets are different. More precisely, the result of Bhowmick and Roche-Newton \cite{Bhowmick/Roche-Newton} in the plane states the following.

\begin{theorem} \label{BRNarcs} The number of arcs in $\mathbb{F}_q^2$ is at most  $2^{(1+o(1))q}$.
\end{theorem}

Observe that this asymptotically matches the trivial lower bound $2^{q+1}$, obtained by considering all possible subsets of an arc of size $q+1$, such as a non-degenerate conic. It is worth noting that the arguments of \cite{Bhowmick/Roche-Newton} to derive Theorem \ref{BRNarcs} are valid for general (Desarguesian and non-Desarguesian) affine and projective planes.
With the same techniques, they can also show that the number of arcs of a fixed size $m$ in $\mathbb{F}_q^2$ is at most ${(1+o(1))q \choose m}$ for $m \in [q^{2/3+o(1)},q]$, which again matches the trivial lower bound ${q+1 \choose m}$ asymptotically. Both of these results were recently extended to arcs in higher dimensions by Chen, Liu, Nie and Zeng \cite{CLNZ}. \\

In this paper, we will show results in the same vein for substructures different from arcs. To this end, we will apply the {\em container} method, a method developed for counting independent sets in {\em hypergraphs}. 

Recall that a hypergraph consists of a set of vertices and a set of hyperedges, where hyperedges are subsets of vertices. A hypergraph is {\em $k$-uniform} if every hyperedge has the same number $k$ of vertices; when $k=2$, a hypergraph is called a {\em graph}. An {\em independent set} in a hypergraph is a set of vertices such that no hyperedge is contained in it. The {\em independence number} of a hypergraph is the size of its largest independent set. \\

Many of the interesting substructures within finite geometries can be phrased as independent sets of a certain (hyper)graph. For example, {\em ovoids} and {\em spreads} of polar spaces are known to be independent sets in the {\em collinearity graph} or the {\em dual polar graph} of the polar space respectively. Similarly, one can observe that caps in $\PG(n,q)$ are independent sets in the $3$-uniform hypergraph whose edges are all collinear triples. These objects generalise arcs in the projective plane in a different direction as the one considered in \cite{CLNZ}. \\ 

To the best of our knowledge, this question has not been considered for any other objects in projective or polar spaces other than arcs as mentioned above. We do not aim to give a systematic treatment of every possible substructure of projective or polar spaces that has been defined in the literature, but rather show the strength of the container method for a select few instances. The main difficulty in the settings that we consider is that the largest independent sets are vanishingly small with respect to the number of vertices of the (hyper)graph under consideration. While this is easily overcome in the graph setting, as we will see in \Cref{sec:containergraphs}, we have to be more careful in the hypergraph setting in \Cref{sec:containercaps} in order to get optimal results. For this reason, we will use a rather technical version of the container lemma, as stated by Saxton and Morris \cite{ST16}. This lemma allows us to prove our results in a more straightforward way, as opposed to \cite{Bhowmick/Roche-Newton,CLNZ,Roche-Newton/Warren}, where the iteration of the container lemma has to be set up every time again.  \\

A subset of an independent set remains independent, and many interesting substructures are closed under taking subsets: we can talk about {\em partial} ovoids, partial spreads, etc. Intuitively, we would like to state a result on the number of substructures (arcs, partial spreads, etc..) as follows: let $C$ be a substructure of the largest possible size, then the number of substructures of size $t$ is at most $\binom{(1+\gamma)|C|}{t}$, where $\gamma$ is small, that is, $\gamma=o(1)$ as $q \rightarrow \infty$. These numbers should be considered as functions in terms of $q$: both the number of substructures and the (rounded) binomial coefficient are polynomials in $q$ of a certain degree, and we aim to show that the highest degree of the former is at most the highest degree of the latter. Also recall that we are not counting the number of non-isomorphic independent sets.
It thus makes sense to allow the correction factor $1+\gamma$, as we have to take into account the number of substructures of maximum size. Such a result translates to saying the number of substructures of maximum size is of a lower order than the total number of substructures of size $t$ we can construct as subsets of a fixed one of largest possible size. These counts should be considered asymptotically, and we do not round numbers that are supposed to be integers either up or down (omitting the floor or ceiling functions). This is justified since these rounding errors are negligible to the asymptotic calculations we make. \\

Finally, the authors of \cite{Bhowmick/Roche-Newton,CLNZ,Roche-Newton/Warren} also investigate the problem of finding the largest arcs in random subsets of points. To be precise, one samples points independently and uniformly at random with fixed probability $p$ and investigates how the size of the maximum arc varies as $p$ changes. This is part of a more general theme in combinatorics where classical combinatorial results are extended to sparse random settings. Such questions have received a lot of attention in the past years, with notable breakthrough results by Conlon and Gowers \cite{ConlonGowers} and Schacht \cite{Schacht}. The hypergraph container method has been highly influential with regards to this problem as well, as can already be seen from the pioneering papers by Balogh, Morris and Samotij, and Saxton and Thomason independently \cite{BMS,SaxtonThomason}.

This is no surprise as this problem is intimately related to counting independent sets by using the first moment method. For instance, if we can count the number of arcs of size $m$ in projective space, we can immediately apply the union bound and see if any survive in expectation after we sample points with probability $p$. For arcs in the projective plane in particular, we have results that are tight up to powers of logarithms for the whole range of $p$ by results of \cite{Bhowmick/Roche-Newton,CLNZ}. 

\paragraph{Organisation of the paper and main results}

This paper is organised as follows.
In \Cref{sec:containergraphs} we introduce the {\bf container method for graphs} and use it to give an asymptotically sharp count for:
\begin{itemize}
\item partial ovoids in polar spaces (\Cref{prop:partialovoids});
\item EKR-sets of flags in polar spaces (\Cref{prop:countEKRsets});
\item partial line spreads in $\PG(r-1,q)$ (\Cref{prop:countlinespreads}) and partial plane spreads in $\PG(5,q)$ (\Cref{cor:countplanespreads}).
\end{itemize}
In \Cref{sec:containercaps} we discuss the {\bf container method for 3-uniform hypergraphs} which we use to investigate the number of caps in $\PG(r,q)$. Our two main results are:
\begin{itemize}
\item bounds on the size of the largest caps in a $p$-random subsets of $\PG(r,q)$ (\Cref{thm:randomsubsets}),
\item an asymptotically sharp count for the number of caps in $\PG(3,q)$ (\Cref{cor:countcaps}). 
\end{itemize}

%
%
%

\section{The container method for graphs}\label{sec:containergraphs}

The container method for graphs dates back to the 1980s \cite{KW} and was extended to hypergraphs independently by Balogh, Morris and Samotij \cite{BMS} and Saxton and Thomason \cite{SaxtonThomason}. These results heavily influenced the field of combinatorics; the aforementioned analogues of classical combinatorial results in random settings are nowadays mainly proved using this method. In short, the container method can be used to bound the number of independent sets in a hypergraph. The container lemma ensures that, under certain conditions, there exists a relatively small collection of {\em containers}: this family of vertex subsets have the property that every independent set of the graph is contained in one of the containers. Moreover, for every container, the induced subgraph on this set has a relatively small number of edges.
As such, the number of independent sets of a certain size $m$ in the hypergraph is at most the number of containers given by the container lemma times the number of subsets of size $m$ in a container; the goal is to derive suitable choices for the parameters occuring in the conditions of the container lemma in order to find a good upper bound on the number of containers and the number of vertices in one container. \\

The container lemma for ordinary graphs reads as follows, where we use the notation $e(S)$ for the number of edges whose vertices are contained in $S$.
\begin{lemma}\label{thm:containerforgraphs} {\bf (The container lemma for graphs)} \cite[Lemma 3.1]{KLRS} 
	Let $G$ be a graph on $n$ vertices, let $f$ be an integer and let $R,\beta$ be real numbers where $0 < \beta < 1$. Then, provided that
	\begin{align}
		e^{-\beta f}n \leq R, \label{eq:graphcondition1}
	\end{align}
	and, for every subset $S \subset V(G)$ of at least $R$ vertices, we have 
	\begin{align}
		2e(S) \geq \beta|S|^2, \label{eq:graphcondition2}
	\end{align}
	the number of independent sets of size $m \geq f$, is at most
	\[\binom{n}{f}\binom{R}{m-f}.\]
\end{lemma}
 
The main hurdle in applying the container lemma is the condition given by \Cref{eq:graphcondition2}. Typically, we will want to choose the parameter $R$ not too be much larger than $\alpha(G)$, the independence number of $G$. 

We know that as soon as a vertex set is larger than $\alpha(G)$, the induced subgraph contains at least one edge. However, it turns out that as soon as the size of a subset $S$ crosses the threshold of $\alpha(G)$, the number of edges in the induced subgraph is relatively large; we can write this number as some fixed positive proportion $\beta$ of the total number $\binom{|S|}{2}$ of possible edges of that set.
In the literature, this phenomenon is often referred to as a {\em supersaturation} result.

It is known that in the case of $d$-regular graphs, we can apply Cauchy eigenvalue interlacing \cite{Haemers}, or equivalently the expander mixing lemma \cite{AlonChung}, to prove such supersaturation statements. We will derive this result in Corollary \ref{cor:supersaturationgraphs}.

\begin{lemma}\label{lem:lowerboundedges}
	Let $G$ be a $d$-regular graph on $n$ vertices and with eigenvalues (of its adjacency matrix) $d = \lambda_1 \geq \lambda_2 \geq \dots \geq \lambda_n$. Then for any subset $S \subset V(G)$ we have
	\[2e(S) \geq \frac{d-\lambda_n}{n}|S|^2+\lambda_n|S|.\]
\end{lemma}

\begin{proof}
The proof is based on eigenvalue interlacing, where we only need the following result: {\em if $V(G)$ is partitioned as $X_1 \cup X_2$ then we can define a quotient matrix $Q$ indexed by these two parts, where $Q_{ij}$ equals the average degree from $X_i$ to $X_j$. The smallest eigenvalue $\mu$ of $Q$ then satisfies $\lambda_2 \geq \mu \geq \lambda_n$}. Recall that $\lambda_n<0$. We refer the reader to \cite{Haemers} for more details on the technique of interlacing.

	Partition the vertex set $V(G)$ as $S \cup (V(G)\setminus S)$. Then the average degree in $S$, $d(S,S)$, equals $2e(S)/|S|$. By recalling that $G$ is $d$-regular, we obtain that the average degree from $S$ to $V(G)\setminus \{S\}$, is $$d(S,V(G)\setminus\{S\})= d-\frac{2e(S)}{|S|}.$$ Since $$|S|d(S,V(G)\setminus \{S\})=(n-|S|)d(V(G)\setminus \{S\},S),$$ we find that $d(V(G)\setminus \{S\},S)=\frac{|S|}{n-|S|}\left(d-\frac{2e(S)}{|S|}\right)$, and finally that $d(V(G)\setminus \{S\},V(G)\setminus \{S\})=d-\frac{|S|}{n-|S|}\left(d-\frac{2e(S)}{|S|}\right)$.

	This gives us the following quotient matrix with respect to the partition $\{S,V(G)\setminus S\}$:
	\[\begin{pmatrix}
	\frac{2e(S)}{|S|} & d-\frac{2e(S)}{|S|} \\
	\frac{|S|}{n-|S|}\left(d-\frac{2e(S)}{|S|}\right) & d-\frac{|S|}{n-|S|}\left(d-\frac{2e(S)}{|S|}\right)
	\end{pmatrix}.\]
	Since the row sums equal $d$, we see that this matrix has eigenvalue $d$. Its other eigenvalue, say $\mu$, can be computed by observing that the trace of the matrix equals $d+\mu$. We find that 
	\begin{align}\mu=\frac{2e(S)}{|S|}-\frac{|S|}{n-|S|}(d-\frac{2e(s)}{|S|}).\label{mu}\end{align}
	On the other hand, by interlacing we know that
	\[\lambda_2 \geq \mu \geq \lambda_n.\]
	Using the second inequality in combination with the expression for $\mu$ found in \eqref{mu}, we obtain the stated inequality.
\end{proof}

We remark that putting $e(S)=0$ in \Cref{lem:lowerboundedges}, we retrieve the well-known Delsarte-Hoffman bound for the independence number $\alpha(G)$ in a regular graph (see e.g. \cite{Haemers2} for an overview of this bound and its history): 
\begin{align}\label{cor:Hoffman}
	\alpha(G) \leq \frac{-\lambda_n}{d-\lambda_n}n.
\end{align}

\begin{cor}\label{cor:supersaturationgraphs}
	Let $G$ be a $d$-regular graph on $n$ vertices with smallest eigenvalue $\lambda$. For all $\varepsilon > 0$, any set $S \subseteq V(G)$ of size at least $(1+\varepsilon)\frac{-n\lambda}{d-\lambda}$ satisfies
	\[2e(S) \geq \left(\frac{\varepsilon}{1+\varepsilon}\right)\frac{(d-\lambda)}{n}|S|^2.\]
\end{cor}
\begin{proof}
	We can rewrite \Cref{lem:lowerboundedges} with $\lambda_n = \lambda$ as 
	\begin{align*}
	2e(S) &\geq |S|^2\left(\frac{d-\lambda}{n}+\frac{\lambda}{|S|}\right) \\
	&\geq |S|^2\left(\frac{d-\lambda}{n}-\frac{1}{1+\varepsilon}\frac{d-\lambda}{n}\right), 
	\end{align*}
	where we used that $\lambda$ is negative and $\frac{1}{|S|}\leq \frac{1}{1+\varepsilon}\frac{d-\lambda}{-n\lambda}$. The lemma follows.
\end{proof}

We are now ready to state a general theorem on the number of independent sets in a $d$-regular graph. We note that similar results are known in the literature, by Alon-R\"odl via the theory of pseudorandom graphs \cite{alon} or similarly via containers by Alon-Balogh-Morris-Samotij \cite{ABMS}, but none of them seem readily applicable to obtain bounds of the form $\binom{(1+o(1))\alpha(G)}{m}$ for the number of independent sets of size $m$, where $m$ only depends on $q$.

\begin{theorem}\label{cor:boundgraphs}
	Let $G$ be a $d$-regular graph on $n$ vertices and with smallest eigenvalue $\lambda$. For all $0<\varepsilon \leq n/d$, let $f = \left(1+\frac{1}{\varepsilon}\right)\frac{n}{d-\lambda}\ln \left(\left(\frac{1}{1+\varepsilon}\right)\frac{d-\lambda}{-\lambda}\right)$ and $\alpha = \frac{-n\lambda}{d-\lambda}$. The number of independent sets of size $m \geq f$ is at most
	\begin{align}\label{boundgraphs}\binom{n}{f}\binom{(1+\varepsilon)\alpha}{m-f}.\end{align}
\end{theorem}

\begin{proof}
	We will apply \Cref{thm:containerforgraphs} to $G$. Define
	\[R = (1+\varepsilon)\frac{-n\lambda}{d-\lambda} = (1+\varepsilon)\alpha \hspace{1cm} \beta = \frac{\varepsilon}{1+\varepsilon}\frac{d-\lambda}{n}.\]
	We see that $\beta=\frac{\varepsilon}{1+\varepsilon}\frac{d-\lambda}{n}\leq \frac{\varepsilon}{1+\varepsilon}\frac{d+n}{n}<1$ since $1+\frac{d}{n}<\frac{1+\varepsilon}{\varepsilon}=1+\frac{1}{\varepsilon}$.

	The condition $e^{-\beta f}n \leq R$ is readily verified with these expressions for $R,\beta,f$ and the condition $2e(S) \geq \beta|S|^2$ for all sets $S$ of size at least $R$ holds by \Cref{cor:supersaturationgraphs}.
\end{proof}

Recall that by the Delsarte-Hoffman bound \eqref{cor:Hoffman}, an independent set in a $d$-regular graph $G$ with smallest eigenvalue $\lambda$ has size at most $\frac{-n\lambda}{d-\lambda}$. In graphs arising from finite geometry, this bound is often sharp and the maximal independent sets are of great interest. Examples include ovoids and spreads in polar spaces. \Cref{cor:boundgraphs} can thus be interpreted as saying that almost all sufficiently large independent sets are subsets of maximal independent sets, and moreover that the overall number of maximal independent sets is not too large. \\

We will now apply \Cref{cor:boundgraphs} to some particular families of graphs arising from finite geometries defined over a finite field $\Fq$. Our goal for each of these cases is to upper bound the number of independent sets of size $m$ by $\binom{(1+o(1))\alpha}{m}$ as $q \to \infty$, where $\alpha = -n\lambda / (d-\lambda)$. Of course, $m$ will need to be sufficiently large to obtain such a result, since a greedy approach shows that the number of independent sets of size $m = \frac{n}{2(d+1)}$ is at least
\[\frac{n(n-(d+1))(n-2(d+1))\dots(n-(m-1)(d+1))}{m!} > \left(\frac{n}{2m}\right)^m = (d+1)^m,\]
which can be much larger than $\binom{(1+o(1))\alpha}{m}$. 

The next result shows that in the finite geometries we consider, we can require $m$ to only be slightly larger than $n/d$, being a factor of a power of $\ln n$ away from it. However, we expend no effort in optimising the exponent of this factor.

\begin{cor}\label{cor:estimationgraphs}
	Let ${\cal G}_q$ be a family of $d$-regular graphs on $n$ vertices and with smallest eigenvalue $\lambda$ defined for infinitely many prime powers $q$, where $n,d,\lambda \to \infty$ as $q \to \infty$.
	
	Let $\alpha := \frac{-\lambda n}{d-\lambda}$ as before, then the number of independent sets of size $m \geq 2\frac{n}{d}\ln^{4}n$, is at most $\binom{(1+o(1))\alpha}{m}$ as $q \to \infty$.
\end{cor}

\begin{proof}
	The crux of the proof lies in showing that the second factor in \eqref{boundgraphs} dominates the first for suitably large values of $m$. Let $0 < \varepsilon \leq n/d$ and $f$ be as in \Cref{cor:boundgraphs} and suppose $m \geq f\left(\frac{\ln n}{\ln(1+\varepsilon)}+1\right)$. Then by \Cref{cor:boundgraphs}, the number of independent sets of size $m \geq f\left(\frac{\ln n}{\ln(1+\varepsilon)}+1\right)$ is at most $\binom{n}{f}\binom{(1+\varepsilon)\alpha}{m-f}$.
	
	Since $m \geq f\left(\frac{\ln n}{\ln(1+\varepsilon)}+1\right)$, we have that $f\left(\ln n+\ln(1+\varepsilon)\right)\leq m \ln(1+\varepsilon)$, which implies that $f\ln n\leq (m-f)\ln(1+\varepsilon)$ and hence $n^f\leq (1+\varepsilon)^{m-f}$.
	We see that 
	\begin{align*}
	\binom{n}{f}\binom{(1+\varepsilon)\alpha}{m-f} &\leq n^f \cdot\binom{(1+\varepsilon)\alpha}{m-f} \\
	&\leq (1+\varepsilon)^{m-f}\cdot\binom{(1+\varepsilon)\alpha}{m-f} \\
	&\leq \binom{(1+\varepsilon)^2\alpha}{m-f},
	\end{align*}
	where we used $\binom{a}{b} \leq a^b$ and {$c^{m-f}\binom{a}{m-f} \leq \binom{ca}{m-f}$} for $c > 1$. 

	Moreover if $m \leq (1+\varepsilon)^2\alpha/2$, we immediately find $\binom{(1+\varepsilon)^2\alpha}{m-f} \leq \binom{(1+\varepsilon)^2\alpha}{m}$. When $m \geq \alpha/2$, we can see that $m/(m-f) < 1+\varepsilon$ for sufficiently large $q$ since $f = o(\alpha)$ and hence $f = o(m)$. Using $\binom{ca}{cb} > \binom{a}{b}$ for $ c> 1$, it then follows that 
	\[\binom{(1+\varepsilon)^2\alpha}{m-f} < \binom{\frac{m}{m-f}(1+\varepsilon)^2\alpha}{m} < \binom{(1+\varepsilon)^3\alpha}{m}.\]
	
	Now we will estimate the expression $f\cdot\left(\frac{\ln n}{\ln(1+\varepsilon)}+1\right)= \left(1+\frac{1}{\varepsilon}\right)\frac{n}{d-\lambda}\ln \left(\left(\frac{1}{1+\varepsilon}\right)\frac{d-\lambda}{-\lambda}\right) \cdot \left(\frac{\ln n}{\ln(1+\varepsilon)}+1\right)$ and set $\varepsilon = (\ln n-1)^{-1}$. We will use the following inequalities for each of the respective factors: 
	\[1+\frac{1}{\varepsilon} = \ln n, \hspace{.5cm} \frac{n}{d-\lambda}\leq \frac{n}{d}, \hspace{.5cm} \frac{1}{1+\varepsilon}<1, \hspace{.5cm} \frac{d-\lambda}{-\lambda}\leq n, \hspace{.5cm} \frac{\ln n}{\ln(1+\varepsilon)}+1\leq 2 \ln^{2 n},\]
	where we respectively used the expression for $\varepsilon$, $\lambda < 0$, $\varepsilon > 0$, the fact that $1 \leq \alpha(G) \leq \alpha$ and $x/2 \leq \ln(1+x)$ for $0<x<1$. We can combine this to find	
	\begin{align*}\left(1+\frac{1}{\varepsilon}\right)\frac{n}{d-\lambda}\ln \left(\left(\frac{1}{1+\varepsilon}\right)\frac{d-\lambda}{-\lambda}\right) \cdot \left(\frac{\ln n}{\ln(1+\varepsilon)}+1\right)&\leq \ln n\cdot \frac{n}{d}\cdot \ln n\cdot 2\ln^{2} n\\
	&= 2\cdot \frac{n}{d} \cdot\ln^{4} n.\end{align*}
	\end{proof}	

Note that the same argument can be made for $m\geq \frac{n}{d}\ln^2n\omega^2(n)$, where $\omega(n)$ is any function tending to infinity when $n\to \infty$, by setting $\varepsilon = (\omega(n))^{-1}$.

\subsection{Partial ovoids in polar spaces}

We will only consider polar spaces which are embedded in a projective space. A polar space $\mathcal{P}$ can be thought of as a point-line incidence geometry where the {\em rank} of this geometry is the (vector) dimension of the largest subspace totally contained in the point set of $\mathcal{P}$. Subspaces of maximal rank contained in $\mathcal{P}$ are called {\em generators}. Consider a subspace $\pi$ of next-to-maximal rank contained in $\mathcal{P}$, and let $s+1$ be the number of generators through $\pi$. The number $s+1$ is independent of the choice of $\pi$ and $s$ is called the {\em type} of the polar space. If $\mathcal{P}$ is embedded in a projective space of order $q$, the type $s$ of $\mathcal{P}$ is of the form $q^t$ where $t\in \{0,1/2,1,3/2,2\}$.

We let $\mathrm{PS}(r,t,q)$ denote a polar space of rank $r \geq 2$, type $q^t$ and defined over a finite field of order $q$. 
The first graph we will consider is the collinearity graph of $\mathrm{PS}(r,t,q)$. The vertices of this graph are given by the points of the polar space, and two vertices are adjacent if and only if the corresponding points are collinear. The collinearity graphs of polar spaces are well-known to be strongly regular. 

The parameters of the collinearity graph are as follows (see e.g. \cite[Theorem 2.2.12]{Brouwer-VM}).
\[n = \frac{q^r-1}{q-1}\left(q^{r+t-1}+1\right), \hspace{1cm} d = \frac{q^{r}-q}{q-1}\left(q^{r+t-2}+1\right), \hspace{1cm} \lambda = -q^{r+t-2}-1.\]

A {\em partial ovoid} in a polar space $\mathcal{P}$ is a set of points $\mathcal{O}$ such that no generator contains more than one point of $\mathcal{O}$. Since every line contained in $\mathcal{P}$ is contained in (at least) a generator, it follows that no two points of a partial ovoid are collinear (i.e. contained in a line of $\mathcal{P}$). Vice versa, a set $\mathcal{O}$ of mutually non-collinear points will have the property that no generator contains more than one point of $\mathcal{O}$. In other words, partial ovoids are exactly the independent sets in the collinearity graph of the polar space.

A direct application of the Delsarte-Hoffman bound \eqref{cor:Hoffman} gives the following well-known upper bound on the size of a partial ovoids:
\begin{cor} Let $\mathcal{O}$ be a partial ovoid of $\mathrm{PS}(r,t,q)$, then 
	\[ |\mathcal{O}|\leq q^{r+t-1}+1.\]
\end{cor}

Moreover, partial ovoids attaining this upper bound (which are also called {\em ovoids}) are only known to exist for small $r$ (see e.g. \cite[Section 3.1]{JanChapter}). If for a polar space $\mathcal{P}$, ovoids do exist, then the number of partial ovoids in $\mathcal{P}$ of size $m$ is clearly at least $\binom{q^{r+t-1}+1}{m}$. We will now prove an asymptotically matching upper bound using \Cref{cor:boundgraphs} for sufficiently large $m$. 


\begin{prop}\label{prop:partialovoids}
	Let $\mathrm{PS}(r,t,q)$ be a polar space, then the number of partial ovoids of size $m \geq c_{r,t}q\ln^{4} q$, is at most 
	\[\binom{(1+o(1))(q^{r+t-1}+1)}{m},\]
	where $c_{r,t}$ is a constant only depending on $r$ and $t$ and can be taken to be $4(2r+t-1)^{4}$.
\end{prop}

\begin{proof} We have that 
\begin{align*}
2\cdot\frac{n}{d}\cdot\ln^{4}n &=2\cdot \frac{\frac{q^r-1}{q-1}(q^{r+t-1}+1)}{\frac{q^r-q}{q-1}(q^{r+t-2}+1)}\cdot\ln^{4}\left(\frac{q^r-1}{q-1}(q^{r+t-1}+1) \right)\\
&\leq 4q\ln^{4}(q^{2r+t-1})\\
&=4(2r+t-1)^{4}q\ln^{4}q,
\end{align*}
where we have used that $\frac{\frac{q^r-1}{q-1}(q^{r+t-1}+1)}{\frac{q^r-q}{q-1}(q^{r+t-2}+1)}<2q$ for sufficiently large $q$.
The proposition now follows from \Cref{cor:estimationgraphs}.

	
%
\end{proof}

\subsection{EKR-sets of flags in polar spaces}

Our next example will come from independent sets in oppositeness graphs in polar spaces. These have been studied in \cite{DBMM} in the context of Erd\H{o}s-Ko-Rado results in polar spaces. Their definition is as follows.
A {\em flag} $\cal F$ in a polar space $\cP$ is a set of subspaces ${\cal F} = \{F_1,\dots,F_k\}$ such that $F_i \lneq F_{i+1}$ for $i = 1,\dots,k-1$. The {\em type} of a subspace of $\cP$ is its vector dimension, and the {\em type} of a flag is the set of types of the subspaces it contains. Two subspaces in $\cP$ are {\em opposite} if no point in the first is collinear to all points of the latter and vice versa. Two flags of the same type in $\cP$ are opposite if all pairs of subspaces of the same type are opposite.

\begin{defin}
	Consider $\cP = \mathrm{PS}(r,t,q)$ and let $J \subseteq [r]$, then the \textbf{oppositeness graph} $G_J$ has all flags in $\cP$ of type $J$ as vertices and two vertices are adjacent if the corresponding flags are opposite.
\end{defin}

For simplicity we will only consider maximal flags, i.e.\ flags of type $[r]$ in a polar space $\mathrm{PS}(r,t,q)$, but analogous results can be deduced for any possible type. The parameters of the oppositeness graph on maximal flags have been computed \cite{DBMM} and are

\[n = \prod_{i=1}^{r}(q^{r+t-i}+1)\qbinom{i}{1}, \hspace{1cm} d = q^{r(r+t-1)} \hspace{1cm} \lambda = \min\{-q^{(r-1)(r+t-1)},(-1)^nq^{r(r-1)}\}.\] 

Independent sets in this graph are sets of pairwise non-opposite flags. They are called {\em Erd\H{o}s-Ko-Rado sets of flags} inspired by the famous result due to Erd\H{o}s, Ko and Rado \cite{EKR} on intersecting families of subsets. If $t \geq 1$ or $n$ is even, we find by the Delsarte-Hoffman bound \eqref{cor:Hoffman} that for an EKR-set of flags $C$ we have 
\[|C| \leq \qbinom{r}{1}\prod_{i=1}^{r-1}(q^{r+t-i}+1)\qbinom{i}{1} =: \alpha_{\rm{EKR}},\]
and this is sharp as shown by the set of maximal flags whose generators all contain a fixed point. 

When $t = 1/2$ and $r$ is odd, it is not known whether the bound, using $\lambda = -q^{r(r-1)}$, is sharp. For $t = 0$ and $r$ odd, one can compute the bound and see that it is also sharp, we refer the interested reader to \cite{DBMM} for more information. We conclude that in the majority of the cases, we can find a lower bound $\binom{\alpha_{\rm EKR}}{m}$ for the number of independent sets of size $m$ in the oppositeness graph. We will now show that this bound is asymptotically sharp, whenever $t \geq 1$ or $r$ is even.

\begin{prop}\label{prop:countEKRsets}
	Let $\mathrm{PS}(r,t,q)$ be a polar space, where $t \geq 1$ or $n$ even, then the number of EKR-sets of maximal flags of size $m \geq c_{r,t}\ln^{4}q$, is at most 
	\[\binom{(1+o(1))\alpha_{\rm{EKR}}}{m},\]
	where $c_{r,t}$ is a constant depending only on $r$ and $t$ and can be taken to be $64r^{4}(r+t-1)^{4}$.
\end{prop}

\begin{proof}
	As before, we have that for sufficiently large $q$
	\begin{align*}
	2\cdot\frac{n}{d}\cdot\ln^{4} n &\leq  2\cdot \frac{2q^{r(r+t-1)}}{q^{r(r+t-1)}}\cdot\ln^{4}\left(2q^{r(r+t-1)} \right)\\
	&\leq 4\cdot 2^4r^{4}(r+t-1)^{4}\ln^{4} q,
	\end{align*}
	using $\ln(2q^{r(r+t-1)}) \leq 2\ln(q^{r(r+t-1)})$.
	The proposition now follows again from \Cref{cor:estimationgraphs}.
	
%
%
%
\end{proof}

\subsection{Partial spreads in projective spaces}

A \textit{partial spread} in a projective space $\PG(r-1,q)$ is a set of pairwise disjoint $k-1$-dimensional spaces. A \textit{spread} is a partial spread that moreover partitions the point set of $\PG(r-1,q)$. It is known that spreads exist if and only if $k$ divides $r$ and clearly they have size $({q^r-1})/({q^k-1})$. When $k = 2$, we can deduce the number of partial line spreads from our previous results. For in this case, the graph whose vertices are the lines of $\PG(r-1,q)$, two lines adjacent whenever they intersect, is a strongly regular graph whose eigenvalues are well-understood \cite[Section 1.2.4]{Brouwer-VM}. For this graph we have 
\[n = \qbinom{r}{2}, \hspace{1cm} d = (q+1)\left(\qbinom{r-1}{1}-1\right), \hspace{1cm} \lambda = -q-1.\]

The Delsarte-Hoffman bound thus gives $\alpha = (q^{r}-1)/(q^2-1)$, which is indeed attained by line spreads whenever $r$ is even. For odd $r$, the results of \cite{NastaseSissokho} show the existence of partial spreads of size $(1+o(1))q^{r-2}$. The upper bound in the following proposition is hence asymptotically sharp for both even and odd $r$.

\begin{prop}\label{prop:countlinespreads}
	For all $r \geq 4$, the number of partial line spreads in $\PG(r-1,q)$ of size $m \geq c_rq^{r-3}\ln^4 q$ is at most
	\[\binom{(1+o(1))\frac{q^{r}-1}{q^2-1}}{m},\]
	where $c_r$ can be taken to be $2^{10}(r-2)^4$.
\end{prop}
\begin{proof}
	As before, we have that for sufficiently large $q$
	\begin{align*}
		2\cdot\frac{n}{d}\cdot\ln^{4} n &\leq  2\cdot \frac{2q^{2r-4}}{q^{r-1}}\cdot\ln^{4}\left(2q^{2r-4} \right)\\
		&\leq 4\cdot q^{r-3}\cdot 4^4(r-2)^4\ln^{4} q,
	\end{align*}
	using $\ln(2q^{2r-4}) \leq 2\ln(q^{2r-4})$.
	The proposition now follows again from \Cref{cor:estimationgraphs}.
\end{proof}

When $k > 2$, the same estimate does not follow as easily: the graph $G$ with vertices $(k-1)$-spaces where $(k-1)$-spaces are adjacent when they have a non-empty intersection is no longer strongly regular and the eigenvalues do not provide the answer. For example when $r = 6$ and $k = 3$, partial plane spreads in $\PG(5,q)$ have size at most $q^3+1$. However, when we look at the eigenvalues, the results are far worse. The eigenvalues of the complement graph have been studied \cite{Eisfeld} and are 
\[\mu_1 = q^9, \hspace{1cm} \mu_2 = q^4, \hspace{1cm} \mu_3 = -q^3, \hspace{1cm} \mu_4 = -q^6.\]
This means that for the graph $G$ we have 
\[n = \qbinom{6}{3} \sim q^9, \hspace{1cm} d = \qbinom{6}{3}-q^9-1 \sim q^8, \hspace{1cm} \lambda = -q^4-1,\]
so that the Delsarte-Hoffman bound gives $\alpha(G) \lesssim q^5$, which is far worse than the upper bound $q^3+1$ . We can therefore also not deduce any reasonable supersaturation from the eigenvalues.

Nevertheless, we can still prove supersaturation combinatorially, but the results we obtain are not optimal.

\begin{lemma}\label{lem:planes}
	For all $\varepsilon > 0$, let $S$ be a set of planes of size $(1+\varepsilon)(q^3+1)$. Then there are at least $\varepsilon(q^4+q^2+1)$ pairs of intersecting planes.
\end{lemma}
\begin{proof}
	Counting with multiplicity, the set $S$ covers $(1+\varepsilon)\frac{q^6-1}{q-1}$ points so that there are $\varepsilon \frac{q^6-1}{q-1}$ points (again counted with multiplicity) where two distinct planes meet. Since two planes meet in at most a line, i.e.\ $q+1$ points, this implies that there must be at least $\varepsilon\frac{q^6-1}{q^2-1} = \varepsilon (q^4+q^2+1)$ pairs of intersecting planes. 
\end{proof}

In other words, if $G$ denotes the graph on the planes in $\PG(5,q)$, two planes adjacent whenever they intersect non-trivially, we see that for all $S \subset V(G)$ such that $|S| = (1+\varepsilon)(q^3+1)$ we have
\[2e(S) \geq \varepsilon(q^4+q^2+1) \geq \frac{\varepsilon}{(1+\varepsilon)^2}\frac{1}{q^2}|S|^2.\]

\begin{cor}\label{cor:countplanespreads}
	The number of partial plane spreads in $\PG(5,q)$ of size $m \geq 28q^2\ln^4 q$ is at most
	\[\binom{(1+o(1))(q^3+1)}{m}.\]
\end{cor}
\begin{proof}
	For any $\varepsilon > 0$, we can use \Cref{thm:containerforgraphs} directly with $R = (1+\varepsilon)(q^3+1)$, $\beta = \frac{\varepsilon}{(1+\varepsilon)^2q^2}$ and hence
	\[f = \frac{1}{\beta}\ln\left(\frac{n}{R}\right) = \frac{(1+\varepsilon)^2q^2}{\varepsilon}\ln\left(\frac{\qbinom{6}{3}}{(1+\varepsilon)(q^3+1)}\right)\]
	suffices. We conclude that the number of independent sets of size $m \geq f$ is at most 
	\[\binom{\qbinom{6}{3}}{f}\binom{(1+\varepsilon)(q^3+1)}{m-f}.\]

	Now set $\varepsilon = (\ln q)^{-1}$ so that 
	\begin{align*}
		f &\leq 4q^2\ln q \cdot \ln\left(2q^6\right) \\
		&\leq 28q^2\ln^2 q,
	\end{align*}
	where we used that $1<1+\varepsilon \leq 2$ and $\ln(2q^6) \leq 7\ln q$. When $m \geq 28q^2\ln^4 q$, we have $m \geq f\ln^2 q$ and so
	\begin{align*}
		\binom{(1+\varepsilon)R}{m-f} \leq \binom{\frac{m}{m-f}(1+\varepsilon)R}{m} \leq \binom{(1+\varepsilon)^2R}{m},
	\end{align*}
	using again $\binom{a}{b} < \binom{ca}{cb}$ for $c > 1$ and the fact that $m/(m-f) \leq 1+\varepsilon$ for $q$ large enough.
	We conclude that the number of independent sets of size $m \geq 28q^2\ln^4q$ is at most
	\begin{align*}
		\binom{\qbinom{6}{3}}{f}\binom{(1+\varepsilon)^2R}{m} &\leq \qbinom{6}{3}^f\binom{(1+\varepsilon)^2R}{m} \\
		&\leq (2q^9)^{m/\ln^2 q}\binom{(1+\varepsilon)^2R}{m} \\
		&\leq \binom{(1+20\varepsilon)(1+\varepsilon)^2R}{m},
	\end{align*}
	where we used that
	\[(2q^9)^{1/\ln^2q} \leq (q^{10})^{1/\ln^2 q} = e^{10/\ln q} \leq 1+\frac{20}{\ln q},\]
	as $e^x \leq 1+2x$ for sufficiently small $x$. Finally as $\varepsilon = (\ln q)^{-1}$, it follows that $(1+20\varepsilon)(1+\varepsilon)^3(q^3+1) = (1+o(1))(q^3+1)$ as $q \to \infty$.
\end{proof}

\section{The container method for caps}\label{sec:containercaps}

As mentioned in the introduction, the container method for hypergraphs was developed by Balogh, Morris and Samotij \cite{BMS} and Saxton and Thomason \cite{SaxtonThomason}. As far as we are aware, the only application in finite geometry of this method has been to count arcs in the plane by Roche-Newton and Warren \cite{Roche-Newton/Warren}, see also subsequent results by Bhowmick and Roche-Newton \cite{Bhowmick/Roche-Newton}, and to count arcs in higher dimensional spaces by Chen, Liu, Nie and Zeng\cite{CLNZ}. 

Many objects in finite geometry can be phrased as independent sets in appropriate graphs, whereas objects corresponding to independent sets in hypergraphs of higher uniformity are more rare. Arcs and caps are instances of such objects. \\ 

{\em A cap} in $\PG(r,q)$ is a set of points, no three of which are collinear. Elementary counting arguments show that a cap in $\PG(2,q)$ has at most $q+1$ points when $q$ is odd, and at most $q+2$ points when $q$ is even. A cap in $\PG(3,q)$ has at most $q^2+1$ points; these are the only known sharp bounds for the size of a cap in $\PG(r,q)$. When $r \geq 4$, the situation is different. In this case, an elementary argument shows that a cap can have at most $\qbinom{r}{1}+1 = (1+o(1))q^{r-1}$ points, but it is a long-standing open problem to determine the (non)-existence of caps in $\PG(r,q)$ of size $\Omega(q^{r-1})$. \\

In this section, we will use the hypergraph $\cH_q$ whose vertices are the points of $\PG(r,q)$, $r \geq 3$, and whose hyperedges are given by collinear triples. Then $\cH_q$ is a $3$-uniform hypergraph on $\qbinom{r+1}{1}$ vertices and has $\frac{1}{6}\qbinom{r+1}{1}\cdot(\qbinom{r+1}{1}-1)\cdot(q-1) = \Theta(q^{2r+1})$ hyperedges. It should be clear that an independent set in $\mathcal{H}_q$ corresponds to a cap of $\PG(r,q)$ and vice versa. \\

\textbf{Notation.} We denote by $V(\cH)$ and $E(\cH)$ the vertex and edge set of a hypergraph $\cH$ respectively, and $v(\cH)$ and $e(\cH)$ their respective sizes. Moreover we denote by $\Delta_\ell(\cH)$ the maximum degree of a set of $\ell$ vertices of $\cH$, that is,
\[\Delta_\ell(\cH) = \max\{d_\cH(A) : A \subset V(\cH), |A| = \ell \},\]
where $d_\cH(A) = |\{B \in E(\cH): A \subset B\}|$. The average degree is denoted by $d(\cH)$ and the subhypergraph of $\cH$ induced by a $S \subset V(\cH)$ is denoted by $\cH[S]$. If $A$ is a set, then $\cP(A)$ denotes its power set. \\

The following container theorem is due to Saxton and Thomason \cite[Theorem 5.1]{ST16}, which we only state for $3$-uniform hypergraphs in order to reduce the complexity.

\begin{theorem}\label{thm:saxtonthomason}
	Let $\cH$ be a $3$-uniform hypergraph on $n$ vertices and let $e_0 \leq e(\cH)$. Let $\tau:\cP(V(\cH))\mapsto \mathbb{R}$ be a function satisfying for each $U \subset V(\cH)$ with $e(\cH[U]) \geq e_0$, that $\tau(U) < 1/2$ and moreover 
	\begin{align}\label{eq:codegreecondition}
	\frac{\Delta_2(\cH[U])}{d(\cH[U]) \cdot \tau(U)}+\frac{1}{2d(\cH[U])\cdot \tau(U)^2} \leq \frac{1}{288}.
	\end{align}
	For $e_0 \leq m \leq e(\cH)$ define
	\begin{align*}
	f(m) &= \max\{|U|\tau(U)\ln (1/\tau(U)) \, : \, U \subset V(\cH), e(\cH[U]) \geq m\} \\
	\tau^* &= \max\{\tau(U) \, : \, U \subset V(\cH), e(\cH[U]) \geq e_0\}.
	\end{align*}
	Let $s = \ln(e(\cH)/e_0)/\ln(12/11)$. Then there exists a collection $\cC \subset \cP(V(\cH))$, a function $b:\cP(V(\cH)) \mapsto \cC$ and a constant $c$ such that 

	\begin{enumerate}
		\item for every independent set $I$ there exists $F \subset I$ with $I \subset b(F) \in \cC$ and $|F| \leq c(s+1)\tau^*n$,
		\item $e(\cH[C]) \leq e_0$ for all $C \in \cC$,
		\item $\ln |\cC| \leq c \sum_{0 \leq i < s}f((12/11)^ie_0))$.
	\end{enumerate}
\end{theorem}

%
%
%
\begin{remark}
	A few remarks are in order.
	\begin{enumerate}
		\item The elements of $\cC$ are the \textit{containers}, whereas the sets $F$ are informally referred to as the \textit{fingerprints}, since any independent set contains a unique fingerprint\footnote{This is in fact only true after we fix an ordering of the vertices of $\cH$, which is necessary for the proof of the container theorem.} which determines the container it is contained in. Moreover, the containers only depend on the fingerprints, and not on the actual independent sets they came from.
		
		\item The constant $c$ appearing in the statement of the theorem can be taken to be $c = 62208$, but the precise value of the constant in the conclusion will not matter, only that it is independent of the parameters of $\cH$. The same goes for the constant $11/12$ appearing in the statement.
		
		\item This version of the container lemma is rather technical, but the need for containers of size $o(v(\cH))$ necessitates this. The original container theorems on the other hand (see \cite[Theorem 3.4]{SaxtonThomason} or \cite[Theorem 2.2]{BMS}), provide containers whose size is a constant proportion of $v(\cH)$. The key insight is that after one application of these theorems, we obtain a set of containers $\cC_1$ and then we iterate the argument again on $\cH[C]$ for all $C \in \cC_1$. In this way, after an appropriate number of iterations, we will in fact be able to obtain containers in $\cH_q$ of size $\Theta(q^{r-1})$. Remark that also \cite[Corollary 3.6]{SaxtonThomason} is designed for a constant number of iterations and works optimally when the `shrinkage factor' is constant. This means that the results one obtains from applying this theorem directly to $\cH_q$ (without iteration) are suboptimal. The reason is that the behaviour of the parameters $\Delta_\ell(\cH_q[C])$ changes as as $C$ becomes increasingly smaller, and this needs to be taken into account.		
		\item We will couple the second item in the conclusion, giving an upper bound on the number of edges within a container, with a supersaturation result to obtain an upper bound on the sizes of the containers. This will immediately imply an upper bound on the total number of independent sets of sufficiently large size $m$: for if $|C| \leq \alpha$ for all $C \in \cC$, then this number is at most $|\cC|\binom{\alpha}{m}$.
	\end{enumerate}
\end{remark}

Lemma \ref{lem:manytriples}, whose proof is similar to \cite[Lemma 2.2]{Bhowmick/Roche-Newton}, will provide the required supersaturation result. We first show the following:
\begin{lemma} \label{Jensen} Let $s,y$ be a positive integers and let $w_1,\ldots,w_s$ be positive real numbers with $\sum_{i=1}^{s}w_i=sy-1$. Then
\begin{align} \sum_{i=1}^s \binom{w_i}{2}\geq (s-1)\binom{y}{2}+\binom{y-1}{2}.
\end{align}
\end{lemma}
\begin{proof} 
Consider the function $f:\mathbb{R}\to \mathbb{R}: x\mapsto \binom{x}{2}=\frac{x(x-1)}{2}$, then $f$ is convex, and Jensen's inequality states \[\sum_{i=1}^s \binom{w_i}{2}\geq s\binom{\frac{sy-1}{s}}{2}.\]
The left hand side is an integer, therefore,  $\sum_{i=1}^s \binom{w_i}{2}\geq \lceil s\binom{\frac{sy-1}{s}}{2} \rceil$. It is easy to check that $s\binom{\frac{sy-1}{s}}{2}=\frac{s(y-1/s)(y-1-1/s)}{2}>(s-1)\frac{y(y-1)}{2}+\frac{(y-1)(y-2)}{2}-1$, so $\lceil s\binom{\frac{sy-1}{s}}{2} \rceil=(s-1)\binom{y}{2}+\binom{y-1}{2}$ and the statement follows.
\end{proof}
\begin{lemma}\label{lem:manytriples}
	Let $\mathcal{P}$ be a set of points in $\PG(r,q)$ of size $k\qbinom{r}{1}$ for some constant $k > 1$. Then we have
	\[ e(\cH)\geq \frac{k^2(k-1)}{6}\qbinom{r}{1}^2-\frac{k(k-1)}{3}\qbinom{r}{1}.\]
	Furthermore, if $|\cP| \geq 2\qbinom{r}{1}$ then $e(\cH_q(\cP)) \geq \frac{ |\cP|^3}{15\qbinom{r}{1}}$. 
\end{lemma}
\begin{proof} Recall that every point $Q$ of $\PG(r,q)$ lies on $\qbinom{r}{1}$ lines. We will label the lines through $Q$ as $\ell^Q_1,\ell^Q_2,\ldots$.
	Now let $Q\in \mathcal{P}$. Denoting the number of points $R\neq Q$ of $\mathcal{P}$ on the line $\ell_i^Q$ by $w_j^Q$, we give the point $Q$ weight $$w(Q)=\frac{1}{3}\sum_{i=1}^{\qbinom{r}{1}}{w_i^Q\choose 2}.$$
	We see that $3w(Q)$ is the number of collinear triples in $\mathcal{P}$ containing $Q$, and hence, double counting the set $\{(Q,t)\,:\, Q\in \mathcal{P}\cap t,\, t\in \cT\}$, we find that
	$$3\cT=\sum_{Q\in \mathcal{P}}3w(Q),$$ or equivalently,
	$$\cT=\sum_{Q\in \mathcal{P}}w(Q).$$
	
	Note that \begin{align*}
	\sum_{j=1}^{\qbinom{r}{1}}w_j^Q
	&=\sum_{j=1}^{\qbinom{r}{1}}(\vert \ell_j^Q\cap \mathcal{P}\vert -1) =|\mathcal{P}|-1,
	\end{align*}
	which is independent of the choice of $Q\in \mathcal{P}$.
	
	
	Using Lemma \ref{Jensen}, with $s=\qbinom{r}{1}$ and $y=k$, we find that
	
	$$\cT=\sum_{Q\in \mathcal{P}}w(Q)=\sum_{Q\in \mathcal{P}}\left(\frac{1}{3}\sum_{i=1}^{\qbinom{r}{1}}{w_i^Q\choose 2}\right)\geq |\mathcal{P}| \cdot \frac{1}{3}\left(\frac{k(k-1)}{2}\qbinom{r}{1}-(k-1)\right).$$
	
	For the last part, we need to check that $\frac{1}{3}k\qbinom{r}{1}\left(\frac{k(k-1)}{2}\qbinom{r}{1}-(k-1)\right)\geq k^3\qbinom{r}{1}^2/15$.
	It is not too hard to check that this is equivalent too 
	\[(3k^2-5k)\qbinom{r}{1}-10k+10\geq 2\] which is satisfied if $k\geq 2$.
\end{proof}

\begin{cor}\label{cor:edgestosize}
	Let $C \subset V(\cH_q)$ be such that $e(\cH_q[C]) \leq  \frac{k^2(k-1)}{6}\qbinom{r}{1}^2 -\frac{k(k-1)}{3}\qbinom{r}{1}$ for some $k > 1$, then $|C| \leq k\qbinom{r}{1}$.
	In particular, if $C \subset V(\cH_q)$ be such that $e(\cH_q[C]) \leq 4\qbinom{r}{1}^2 $, then $|C| \leq 4\qbinom{r}{1}$.
\end{cor}

We are now in the position to build our family of containers, restricting ourselves to the case $r \geq 3$ as $r = 2$ was handled in \cite{Bhowmick/Roche-Newton,Roche-Newton/Warren}. Our goal will be to end up with containers of size at most $(1+\gamma)\qbinom{r}{1}$ for some $\gamma > 0$. Depending on the application, we will choose $\gamma$ to be constant or $\gamma = o(1)$. The latter is in general best possible since for $r=3$ there are independent sets of size $q^{r-1}+1$ in $\cH_q$. The proof below is heavily inspired by \cite[Theorem 1.10]{ST16}. 

\begin{theorem}\label{thm:containersHq}
	Let $r \geq 3$ and $\cH_q$ be the 3-uniform hypergraph as defined before.
	For all sufficiently large $q$, there exists a collection $\cC \subset \cP(V(\cH_q))$, a function $b:\cP(V(\cH_q))\to \cC$, and absolute constants $c_1,c_2$ such that
	
	\begin{enumerate}
		\item for every independent set $I$ in $\cH_q$ there exists $F \subset I$ with $I \subset b(F) \in \cC$ and $|F| \leq c_1 q^{(r-1)/2}\ln q$, 
				\item $e(\cH_q[C]) \leq 4\qbinom{r}{1}^2$ for all $C \in \cC$,
		\item $\ln |\cC| \leq c_2 q^{(r-1)/2}\ln^2 q$.
	\end{enumerate}
\end{theorem}

\begin{proof}
	We will apply \Cref{thm:saxtonthomason} to $\cH_q$ with $e_0 = 4\qbinom{r}{1}^2$. So let $U \subset V(\cH_q)$ with $e(\cH_q[U]) =: m \geq e_0$ and denote $|U| = u$. By \Cref{lem:manytriples} we see that if $u \geq 2\qbinom{r}{1}$ then $m \geq u^3/15\qbinom{r}{1}$ or equivalently $u \leq (15\qbinom{r}{1}m)^{1/3}$. In the other case we have
	\[u \leq 2\qbinom{r}{1} = \left(8\qbinom{r}{1}^3\right)^{1/3} \leq \left(15\qbinom{r}{1}m\right)^{1/3},\]
	and so the inequality holds in either case.
	
	Now our objective is to find suitable $\tau(U)$ such that \eqref{eq:codegreecondition} holds. Since $\Delta_2(\cH[U]) \leq q$ for any $U \subset V(\cH_q)$ this condition is
	\begin{align*}
	\frac{qu}{m \cdot \tau(U)}+\frac{u}{2m\cdot \tau(U)^2} \leq \frac{1}{288}.
	\end{align*}
	This inequality is satisfied if we set
	\[\tau(U) = \max\left\{576\frac{qu}{m},18\left(\frac{u}{m}\right)^{1/2}\right\},\]
	as the left-hand side can then be bounded using both lower bounds for $\tau(U)$:
	\begin{align*}
	\frac{qu}{m \cdot \tau(U)}+\frac{u}{2m\cdot \tau(U)^2} \leq \frac{1}{576} + \frac{1}{648} < \frac{1}{288}.
	\end{align*}
	
	This definition of $\tau(U)$ makes checking the condition easy, and it is clear that the left hand side of the inequality will only decrease when we increase any of the constants $576$ and $18$ in the definition of $\tau(U)$. In fact, we will see that by appropriately increasing the second constant, the maximum will always be attained in the second expression. So replacing the constant $18$ by $18c_1$, $c_1 \geq 1$, we can rewrite the condition that the second expression is the largest:
	\begin{align}
	576\frac{qu}{m} \leq 18c_1\left(\frac{u}{m}\right)^{1/2}& \iff 32q\left(\frac{u}{m}\right)^{1/2} \leq c_1\nonumber \\
	&\iff 1024q^2u \leq c_1^2m.\label{e1}
	\end{align}
	Using  $u \leq (15\qbinom{r}{1}m)^{1/3}$ a sufficient condition for \eqref{e1} to be satisfied is that
	\begin{align*}
	1024q^2(15\qbinom{r}{1}m)^{1/3} \leq c_1^2m &\iff 15\cdot 2^{30} q^6\qbinom{r}{1} \leq c_1^6m^2,
	\end{align*} 
	which in its turn would be satisfied, using $m \geq e_0 = 4\qbinom{r}{1}^{2}$, if 
	\begin{align*}
	15\cdot 2^{30}q^6\leq 16c_1^6\qbinom{r}{1}^3.
	\end{align*}
	Since $q^6 \leq \qbinom{r}{1}^3$, using $r \geq 3$, clearly $c_1 = 2^5$ suffices. In conclusion, we can set $\tau(U) = c_2\left(\frac{u}{m}\right)^{1/2}$ for some large constant $c_2$ (e.g. $18\cdot32$) and \eqref{eq:codegreecondition} will be satisfied.
	
	Morever, $(u/m)^3 \leq 15\qbinom{r}{1}/m^{2} \leq (15/16)\cdot \qbinom{r}{1}^{-3}$ can be made arbitrarily small by choosing sufficiently large $q$, so in particular $\tau(U) < 1/2$ for sufficiently large $q$. The conditions of \Cref{thm:saxtonthomason} are hence satisfied, and we obtain a constant $c$ and a set of containers $\cC$ with the properties mentioned in the conclusion of the theorem. We will now calculate what these parameters are. \\
	
	First of all,
	\begin{align*}
	u\tau(U)\ln(1/\tau(U)) &\leq c_2\frac{u^{3/2}}{m^{1/2}}\ln\left(\frac{m^{1/2}}{c_2u^{1/2}}\right) \\
	&\leq c_2\frac{(15\qbinom{r}{1}m)^{1/2}}{m^{1/2}}\ln\left(\frac{m^{1/2}}{c_2(15\qbinom{r}{1}m)^{1/6}}\right) \\
	&\leq c_3q^{(r-1)/2}\ln\left(m^{1/3}q^{-(r-1)}\right) =: f(m),
	\end{align*}
	for sufficiently large $q$, where the first inequality holds since $\tau\ln(1/\tau)$ is an increasing function of $\tau$ when $\tau < 1/e$, the second inequality holds since $u^{3/2}\ln(m^{1/2}/cu^{1/2})$ is an increasing function of $u$ when $u < m/c^2e^{2/3}$, 

	and the third inequality holds for a sufficiently large constant $c_3$. We can then compute the parameter $s$:

	\begin{align*}
	s = c_4\ln\left(\frac{e(\cH_q)}{e_0}\right) \leq c_4\ln\left(\frac{q^{2r+1}}{q^{2r-2}}\right) \leq c_5\ln q,
	\end{align*}
	for some constant $c_5$ and sufficiently large $q$.
	
	It follows that, setting $\alpha = 12/11$,
	\begin{align*}
	\ln|\cC| &\leq \sum_{i=0}^{s}f(\alpha^i e_0) \\
	&\leq c_3q^{(r-1)/2}\sum_{i=0}^{s}\ln\left(\alpha^{i/3}e_0^{1/3}q^{-(r-1)}\right) \\
	&\leq c_6q^{(r-1)/2}\ln^2 q,
	\end{align*}
	for some constant $c_6$, where we used that both $s$ and the terms appearing in the sum are at most a constant times $\ln q$. This proves the bound on the number of containers.
	
	Finally, using the remark after the \Cref{thm:saxtonthomason} on the size of the fingerprints, we can also bound the size of the fingerprints:
	\begin{align*}
		|U|\tau(U) &= c_2\left(\frac{u^3}{m}\right)^{1/2} \\
		&\leq c_2\left(\frac{15\qbinom{r}{1}m}{m}\right)^{1/2} \\
		&\leq c_7q^{(r-1)/2},
	\end{align*}
	for some constant $c_7$, so that $g(m) \leq c_7q^{(r-1)/2}$ for all $m \geq e_0$. It follows that
	\begin{align*}
	|F| &\leq c\sum_{0 \leq i < s}g((12/11)^ie_0) \\
	&\leq c \cdot s \cdot c_7q^{(r-1)/2} \\
	&\leq c_8q^{(r-1)/2}\ln q, 
	\end{align*}
	for some constant $c_8$ as $s = O(\ln q)$.
\end{proof}

\subsection{First application: caps in random subsets of $\PG(r,q)$}

A very influential research line in the last few years has been the study of random Tur\'an numbers. The problem generalises the classical Tur\'an problem which asks for the maximal number of edges in an $H$-free subgraph of $K_n$, for some fixed graph $H$. Instead, we change the host graph to be the binomial random graph $G(n,p)$. Initial groundbreaking results in this setting have been obtained by Conlon and Gowers \cite{ConlonGowers} and Schacht \cite{Schacht} independently. It was later shown that the same results could be obtained by the container method \cite{BMS,SaxtonThomason}, which was one of the major driving factors behind its development. Since then, this problem and generalisations in different settings have been studied extensively, typically stating the problem as finding sharp bounds for the largest independent set in a random sub(hyper)graph of a given (hyper)graph. One such example which is relevant to our work includes the size of the largest arc in a random subset of $\PG(r,q)$ \cite{CLNZ}. \\

The problem statement in the case of caps in $\PG(r,q)$ is the following: denote by $\cH_q(p)$ the random subhypergraph of $\cH_q$, obtained by sampling every point in $\PG(r,q)$ (i.e. every vertex in $\cH_q$) independently and uniformly with probability $p$. What is then the behaviour of the independence number $\alpha(\cH_q(p))$ as $p = p(q)$ varies? With the help of \Cref{thm:containersHq}, we can obtain good bounds. 


We will make use of the following results. First, the following tools from the probabilistic method will be used throughout, see \cite{AS} for a reference. \\

\noindent\textbf{Chernoff Bound} If $X$ is a binomial random variable and $0 < c \leq 1$, then
\[\mathrm{Pr}\left(|X-\mathbb{E}(X)| \geq c\,\mathbb{E}(X)\right) \leq 2\exp\left(\frac{-c^2\,\mathbb{E}(X)}{3}\right).\]

\noindent\textbf{Markov's inequality} If $X$ is a non-negative random variable and $a > 0$ then 
\[\mathrm{Pr}(X \geq a) \leq \frac{\mathbb{E}(X)}{a}.\]

Before we state the result, observe that we can obviously expect $\alpha(\cH_q(p))$ to vary as $p$ changes. When $p$ is small, we expect to see few collinear triples since $\cH_q(p)$ is sparse and hence close to an empty hypergraph. On the other hand, when $p$ is large, the hypergraph $\cH_q(p)$ resembles the original hypergraph $\cH_q$ closely and so we expect $\alpha(\cH_q(p))$ to be the same as $\alpha(\cH_q)$, appropriately scaled by $p$. Interestingly enough, there is a `flat' range where $\alpha(\cH_q(p)) = \widetilde{\Theta}(q^{(r-1)/2})$, suppressing powers of logarithms, where the answer seems to be independent of $p$. This phenomenon has been observed in other related problems as well \cite{CLNZ,SpiroNie}. We write $f \ll g$ for two real non-negative functions if $f = o(g)$.

\begin{theorem}\label{thm:randomsubsets}
	Asymptotically almost surely (a.a.s.), that is, with probability tending to $1$ as $q \to \infty$, we have
	\begin{align*}
		\alpha(\cH_q(p)) =
		\begin{cases}
		(1+o(1))pq^r & q^{-r} \ll p \ll q^{-(r+1)/2} \\
		O(q^{(r-1)/2}\ln^2 q) & q^{-(r+1)/2} \leq p \leq q^{-(r-1)/2}\ln^2 q \\
		O(pq^{r-1}) & q^{-(r-1)/2}\ln^2 q \leq p \leq 1.
	\end{cases}
	\end{align*}

	Moreover, we have a.a.s. lower bounds
	\begin{align*}
		\alpha(\cH_q(p)) =
		\begin{cases}
			\Omega(q^{(r-1)/2}/\ln q) & q^{-(r+1)/2}/\ln q \leq p \leq 1, \\
			\Omega(pq^{r-1}) & q^{-(r-1)/2}\ln q \ll p \leq 1 \text{ when } r \in \{2,3\}.
		\end{cases}
	\end{align*}
\end{theorem}

\begin{proof}
	\textbf{Bounds for small $p$.} Clearly $\alpha(\cH_q(p)) \leq v(\cH_q(p))$ always holds. Moreover, if we delete one vertex from every edge of $\cH_q(p)$, we get an independent set. Hence, if $X$ is the random variable $e(\cH_q(p))$, i.e.\ the number of edges of $\cH_q(p)$ then $\alpha(\cH_q(p)) \geq v(\cH_q(p))-X$. It remains to show that a.a.s.\ $v(\cH_q(p)) = (1+o(1))pq^r$ and $X = o(pq^r)$. 
	
	Observe that $\mathbb{E}(v(\cH_q(p))) = p\qbinom{r+1}{1} = (1+o(1))pq^r$ and since $p \gg q^{-r}$, it follows by Chernoff's bound that a.a.s.\ $v(\cH_q(p)) = (1+o(1))pq^r$. On the other hand $\mathbb{E}(X) = p^3e(\cH_q) = (1+o(1))p^3q^{2r+1}= o(pq^r)$, as $p \ll q^{-(r+1)/2}$. Hence for every $\varepsilon > 0$ we have by Markov's inequality
	\[\mathrm{Pr}(X \geq \varepsilon pq^r) \leq \frac{(1+o(1))p^3q^{2r+1}}{\varepsilon pq^r} \rightarrow 0 \text{ as } q\to\infty.\] 
	And hence we conclude that $X = o(pq^r)$. \\
	
	\textbf{Upper bound for large $p$.} 
	 Let $\cC$ be the family of containers obtained in \Cref{thm:containersHq}, so that $\ln|\cC| \leq c_2q^{(r-1)/2}\ln^2 q$ for some constant $c_2$. Using \Cref{cor:edgestosize}, we see that $|C| \leq 4\qbinom{r}{1}$ for all $C \in \cC$. Denote by $X_i(p)$ the random variable that counts the number of independent sets of size $i$ in $\cH_q(p)$, and set $m = C_0pq^{r-1}$ for some constant $C_0$ to be determined later. Since every independent set is contained in some element of $\cC$, we find by Markov's inequality that

	\begin{align*}
	\mathrm{Pr}(X_m(p) \geq 1) \leq \mathbb{E}(X_m(p)) &\leq \sum_{C \in \cC}p^m\binom{|C|}{m} \\ 
	&\leq |\cC|p^m\binom{8q^{r-1}}{m} \\
	&\leq\exp(c_2q^{(r-1)/2}\ln^2q)\left(\frac{8epq^{r-1}}{m}\right)^m,
	\end{align*}
	where we used $4\qbinom{r}{1} \leq 8q^{r-1}$ for $q$ large enough in the second inequality, and $\binom{a}{b} \leq (ea/b)^b$ in the third.
	Substituting $m = C_0pq^{r-1}$ and using $q^{-(r-1)/2}\ln^2q \leq p$, we conclude that for sufficiently large $C_0$ we have $\mathrm{Pr}(X_m(p) \geq 1) \rightarrow 0$ as $q \rightarrow \infty$. \\

	\textbf{Upper bound for medium $p$.} Now suppose that $q^{-(r+1)/2} \leq p \leq q^{-(r-1)/2}\ln^2 q$, and let $m := C_0q^{(r-1)/2}\ln^2 q$ for some constant $C_0$ to be determined later. By Markov's inequality and the fact that $\mathbb{E}(X_m(p))$ is non-decreasing in $p$, we find that 
	\begin{align*}
	\mathrm{Pr}(X_m(p) \geq 1) \leq \mathbb{E}(X_m(p)) \leq \mathbb{E}(X_m(q^{-(r-1)/2}\ln^2 q)).
	\end{align*}
	
	With the same calculation as before, but now for these values of $m$ and $p$, we find again that $\mathrm{Pr}(X_m(p) \geq 1) \to 0$ as $q \to \infty$ in this range for $p$. \\
	
	\textbf{Lower bound for medium-large $p$.} 
	
	From the monotonicity of $\alpha(\cH_q(p))$ we find for the range $q^{-(r+1)/2}/\ln q \leq p \leq 1$ that a.a.s.
	\[\alpha(\cH_q(p)) \geq \alpha(\cH_q(q^{-(r+1)/2}/\ln q)) = (1+o(1))q^{(r-1)/2}/\ln q,\]
	
	as we already deduced that $\alpha(\cH_q(p)) = (1+o(1))pq^r$ a.a.s.\ for $p = q^{-(r+1)/2}/\ln q$.
	
	In the special case when $r \in \{2,3\}$, we have that $\alpha(\cH_q) = \Theta(q^{r-1})$. Sampling from an independent set of largest possible size, we see that by the Chernoff bound it follows that a.a.s.\ $\alpha(\cH_q(p)) = \Omega(pq^{r-1})$.
\end{proof}

\subsection{Second application: the number of caps in $\PG(3,q)$}

As a second application of the hypergraph container lemma, we will count the number of caps of size $m$ in $\PG(3,q)$ for sufficiently large $m$. A similar count in $\PG(2,q)$ was obtained before \cite{Bhowmick/Roche-Newton,Roche-Newton/Warren}. Even though the method works in general dimension, we focus on the case $r=3$ since this is the only known dimension, other than $r=2$, for which we have an asymptotically matching lower bound. As mentioned earlier, it is a major open problem to prove or disprove the existence of caps of size $\Omega(q^{r-1})$ in $\PG(r,q)$ for $r \geq 4$. 

In order to obtain precise bounds of the form $\binom{(1+o(1))q^2}{m}$ for sufficiently large $m$, we will need to further reduce the size of the containers in \Cref{thm:containersHq}. Fortunately, we only need a constant number of iterations to do this, so we do not need the full strength of the iterated container theorem (\Cref{thm:saxtonthomason}). Instead, we will use \cite[Corollary 1.3]{ST16}, which we again state only for $3$-uniform hypergraphs.

\begin{theorem}\label{thm:saxtonthomason2}
	Let $\cH$ be a hypergraph on $n$ vertices of average degree $d(\cH)$ and $\varepsilon > 0$. Then there is a constant $c = c(\varepsilon)$ for which the following holds. Let $0 <\tau \leq 1$ be chosen so that
	\[\Delta_2(\cH) \leq c\tau d(\cH) \hspace{1cm} \text{and} \hspace{1cm} 1 \leq c\tau^2 d(\cH).\]
	Then there exists a collection $\cC \subset \cP(V(\cH))$ and a constant $c_0$ such that
	\begin{enumerate}
		\item for every independent set $I$ there exists some $C\in \cC$ such that $I \subset C$,
		\item $e(\cH[C]) \leq \varepsilon e(\cH)$,
		\item $\ln |\cC| \leq c_0\tau n \ln(1/\tau)$.
	\end{enumerate}
\end{theorem}

We will now apply this theorem to the containers obtained in \Cref{thm:containersHq} iteratively until we obtain a set of containers $\cC_2$ with the property that each container $C$ satisfies $|C| \leq (1+\gamma)\qbinom{3}{1}$ for some $0<\gamma \leq 1$. 

\begin{theorem}\label{thm:containersHq2}
	Let $\cH_q$ be the 3-uniform hypergraph as defined before. For all $0<\gamma \leq 1/10$ and sufficiently large $q$ we have a family $\cC$ of subsets of $V(\cH)$ satisfying the following properties:
	\begin{itemize}
		\item for every cap $I$ in $\PG(3,q)$, there exists some $C \in \cC$ such that $I \subset C$,
		\item $|C| \leq (1+\gamma)\qbinom{3}{1}$ for all $C \in \cC$,
		\item $\ln|\cC| \leq cq\ln^2 q + c\gamma^{-1} q\ln q$ for some absolute constant $c >0$.
	\end{itemize}
\end{theorem}

\begin{proof}
	Let $\cC$ be the set of containers for $\cH_q$ obtained in \Cref{thm:containersHq}. 
	Note that $\ln|\cC|\leq c_2q\ln^2 q$ and that for every container $C\in \cC$, $e(\cH[C]\leq 4\qbinom{3}{1}^2$, so \Cref{cor:edgestosize} implies that $|C|\leq 4\qbinom{3}{1}$.
	
	We will apply \Cref{thm:saxtonthomason2} to any container $C \in \cC$ for which $e(\cH_q[C]) \geq \frac{(1+\gamma)^2\gamma}{12}\qbinom{3}{1}^2$ to obtain a new set $\cC_2$ of containers. Otherwise, set $\cC_2 = \{C\}$.
	 In the latter case, since $\frac{(1+\gamma)^2\gamma}{12} \qbinom{3}{1}^2<\frac{(1+\gamma)^2\gamma}{6}\qbinom{3}{1}^2-\frac{\gamma(1+\gamma)}{3}\qbinom{3}{1}$, \Cref{cor:edgestosize}  implies that $|C| \leq (1+\gamma)\qbinom{3}{1}$.
	
	For any $0<\gamma < 1/10$ we will apply \Cref{thm:saxtonthomason2} with $\varepsilon = (1+\gamma)^2\gamma/48$. 
	As before, $\Delta_2(\cH_q) \leq q$ and $$d(\cH_q[C])=\frac{3 e(\cH[C])}{|C|}\geq \frac{3 \frac{(1+\gamma)^2\gamma\qbinom{3}{1}^2}{12}}{4\qbinom{3}{1}}\geq \frac{(1+\gamma)^2\gamma}{16}\qbinom{3}{1}.$$ Hence, it would be sufficient to choose $0 < \tau \leq 1$ such that 
	\[q \leq c\tau {\frac{(1+\gamma)^2\gamma}{16}}\qbinom{3}{1} \hspace{.5cm} \text{and} \hspace{.5cm} 1 \leq c\tau^2 {\frac{(1+\gamma)^2\gamma}{16}}\qbinom{3}{1} ,\]
	where $c = c(\varepsilon)$, with $\varepsilon=(1+\gamma)^2\gamma/48$, given by \Cref{thm:saxtonthomason2}. Setting $\tau = \tilde{c} (1+\gamma)^{-2}\gamma^{-1}q^{-1}$, and using that $\gamma<1/10$, we find that both conditions are satisfied for a sufficiently large constant $\tilde{c}$. We conclude the existence of a set of containers $\cC_C$ in $\cH_q[C]$ such that $|C_2| \leq (1+\gamma)^2\gamma\qbinom{3}{1}$ for all $C_2 \in \cC_C$ and $\ln|\cC_2| \leq c'\gamma^{-1} q\ln(q)$ for some constant $c'$, using that $0 < \gamma < 1/10$.
		
	This means that the total number of containers $\cC_2 := \bigcup_{C \in \cC} \cC_C$ satisfies
	\[\ln|\cC_2| \leq c_2q\ln^2 q + c'\gamma^{-1} q\ln(q).\]
	We can replace the constants by $\max\{c_2,c'\}$ to make them equal.
	We find that $e(\cH_q[C_2]) \leq \frac{(1+\gamma)^2\gamma}{12} \qbinom{3}{1}^2$ for all $C_2 \in \cC_2$. As before, it follows from \Cref{thm:containersHq} that $|C_2| \leq (1+\gamma)\qbinom{3}{1}$.		

\end{proof}

\begin{cor}\label{cor:countcaps}
	The number of caps in $\PG(3,q)$ of size $m \geq cq\ln^3q$ for some constant $c > 0$ is at most $\binom{(1+o(1))q^2}{m}$ as $q \to \infty$.
\end{cor}

\begin{proof}
	
	Putting $\gamma = (\ln q)^{-1}$, we find for sufficiently large $q$ by \Cref{thm:containersHq2} a collection of containers $\cC$ such that $|C| \leq (1+\gamma)\qbinom{3}{1}$ and $\ln|\cC| \leq c_0 q\ln^2q$ for some constant $c_0 > 0$. Since every independent set in $\cH_q$, i.e.\ a cap in $\PG(3,q)$, is contained in a container, we find that the number of caps of size $m \geq 2c_0q\ln^3q$  in $\PG(3,q)$ is at most
	\begin{align*}
		\sum_{C \in \cC}\binom{|C|}{m} &\leq \exp({c_0q\ln^2q})\binom{(1+\gamma)\qbinom{3}{1}}{m} \\
		&\leq \exp\left({\frac{m\gamma}{2}}\right)\binom{(1+\gamma)\qbinom{3}{1}}{m} \\
		&\leq (1+\gamma)^m\binom{(1+\gamma)\qbinom{3}{1}}{m} \\
		&\leq \binom{(1+\gamma)^2\qbinom{3}{1}}{m},
	\end{align*}
	as $\gamma/2 \leq \ln(1+\gamma)$ for sufficiently large $q$. Finally, put $2c_0=c$ and observe that $(1+1/\ln q)\qbinom{3}{1} = (1+o(1))q^2$ as $q \to \infty$.
\end{proof}


\bibliographystyle{alpha}

\end{document}